\documentclass[a4paper, 11pt]{amsart}
\pdfoutput=1

\usepackage[utf8]{inputenc}
\usepackage{mathtools,amssymb}
\usepackage{enumerate}
\usepackage{hyperref}
\usepackage{xcolor,graphicx}
\usepackage{mathrsfs}
\usepackage{faktor}
\usepackage{enumitem}

\mathtoolsset{showonlyrefs}

\graphicspath{{images/}}

\newtheorem{theorem}{Theorem}[section]
\newtheorem{lemma}[theorem]{Lemma}

\newtheorem{remark}[theorem]{Remark}

\title[Boundary rigidity for Randers metrics]{Boundary rigidity for Randers metrics}
\keywords{Inverse problems, boundary rigidity, travel time tomography.}
\subjclass[2020]{53C24, 53A35, 86A22}

\author{Keijo M\"onkk\"onen}
\thanks{Department of Mathematics and Statistics, University of Jyv\"askyl\"a, P.O. Box 35 (MaD) FI-40014 University of Jyv\"askyl\"a, Finland; \href{mailto:kematamo@student.jyu.fi}{kematamo@student.jyu.fi}}

\date{\today}

\newcommand{\R}{{\mathbb R}}

\newcommand{\der}{{\mathrm d}}
\newcommand{\id}{\mathrm{Id}}

\newcommand{\aabs}[1]{\left\lVert #1 \right\rVert}

\begin{document}
\maketitle
\begin{abstract}
If a non-reversible Finsler norm is the sum of a reversible Finsler norm and a closed 1-form, then one can uniquely recover the 1-form up to potential fields from the boundary distance data. We also show a boundary rigidity result for Randers metrics where the reversible Finsler norm is induced by a Riemannian metric which is boundary rigid. Our theorems generalize Riemannian boundary rigidity results to some non-reversible Finsler manifolds. We provide an application to seismology where the seismic wave propagates in a moving medium.
\end{abstract}

\section{Introduction}
In this article we study a certain type of inverse problem for a special class of Finsler norms. The inverse problem we consider is known as the boundary rigidity problem: does the boundary distance data determine the Finsler norm uniquely up to the natural gauge in question? Here we present the problem and our results in a general level; more detailed information can be found in sections~\ref{sec:mainresults}, \ref{sec:applications} and~\ref{sec:preliminaries}.

 Let~$M$ be a smooth manifold with boundary~$\partial M$. A Finsler norm~$F$ on~$M$ is a non-negative function on the tangent bundle $F\colon TM\rightarrow[0, \infty)$ such that for each $x\in M$ the map $y\mapsto F(x, y)$ defines a positively homogeneous norm in~$T_xM$. In general, Finsler norms are homogeneous only in positive scalings and they induce a distance function on~$M$ which is not necessarily symmetric in contrast to the Riemannian distance function. 

Let~$\beta$ be a smooth 1-form on~$M$ and~$F_r$ a reversible Finsler norm, i.e. $F_r(x, -y)=F_r(x, y)$ for all $x\in M$ and $y\in T_xM$. If the norm of $\beta$ with respect to~$F_r$ is small enough, we can define the non-reversible Finsler norm $F=F_r+\beta$. The Finsler norm~$F$ is non-reversible in the sense that $F(x, -y)=F(x, y)$ for all $x\in M$ and $y\in T_xM$ if and only if $\beta=0$. We can thus think that~$\beta$ is an anisotropic perturbation to the reversible Finsler norm~$F_r$. We further assume that~$\beta$ is closed ($\der\beta=0$) which implies that~$F$ and~$F_r$ have the same geodesics as point sets and that~$F$ has reversible geodesics.

Suppose we know the boundary distance data of~$F=F_r+\beta$, i.e. we know the lengths of all geodesics of~$F$ connecting two points on the boundary~$\partial M$. The question is: can we say something about~$\beta$ and~$F_r$ from this information? We prove that if~$M$ is simply connected, then one can uniquely recover the 1-form~$\beta$ (up to potential fields) and the boundary distance data of~$F_r$ from the boundary distance data of~$F$ (see theorem~\ref{thm:sinjectivityfromboundarydistances} for a precise statement).

Riemannian metrics form a special class of reversible Finsler norms. Suppose that~$F_r$ is induced by a Riemannian metric~$g$ and write $F_r=F_g$. If $\aabs{\beta}_g<1$, then $F=F_g+\beta$ defines a non-reversible Finsler norm called Randers metric. We say that the Riemannian manifold~$(M, g)$ is boundary rigid, if the boundary distance data determines the metric~$g$ uniquely up to boundary preserving diffeomorphism. We prove that if~$M$ is simply connected and $(M, g)$ is boundary rigid, then $(M, F)$ is also boundary rigid in the sense that one can uniquely recover the 1-form~$\beta$ up to potential fields and the Riemannian metric~$g$ up to boundary preserving diffeomorphism from the boundary distance data of~$F$. See theorem~\ref{thm:maintheorem} for a precise statement.

Our proofs are mainly based on the following two facts. First, if two Finsler norms differ only by a closed 1-form, then they are projectively equivalent (they have the same geodesics modulo orientation preserving reparametrizations). Second, since $F=F_r+\beta$, we can express the length of any curve~$\gamma$ with respect to~$F_r$ in terms of the symmetric part of the length functional~$L_F(\gamma)$. Similarly, the integral $\int_\gamma\beta$ can be expressed in terms of the antisymmetric part of~$L_F(\gamma)$. This allows us to reduce the boundary rigidity problem of~$F$ to the boundary rigidity problem of~$F_r$.

Boundary rigidity has been studied earlier mainly on Riemannian manifolds. Boundary rigidity is known for example for simple subspaces of Euclidean space~\cite{GRO-filling-riemannian-manifolds}, simple subspaces of symmetric spaces of constant
negative curvature~\cite{BCG-rigidity}, conformal simple metrics which agree on the boundary~\cite{CRO-rigidity-conformal, SUVZ-travel-time-tomography} and for certain two-dimensional manifolds including compact simple surfaces~\cite{CRO-rigidity-negative-curvature, RE-boundary-rigidity-conjecture, MU-reconstruction-problem-two-dimensional, PU-simple-manifolds-boundary-rigidity}. It is also conjectured that compact simple manifolds of any dimension are boundary rigid~\cite{RE-boundary-rigidity-conjecture}. Our results generalize the boundary rigidity results to certain Randers metrics whenever the boundary rigidity of the unperturbed Riemannian manifold is known (see theorem~\ref{thm:maintheorem}). For a more comprehensive treatment of the boundary rigidity problem in Riemannian geometry, see the review~\cite{SUVZ-travel-time-tomography}.  

Closest to our main theorems are rigidity results for magnetic geodesics on Riemannian manifolds. In~\cite{DPSU-boundary-rigidity-magnetic} the authors prove boundary rigidity in the presence of a magnetic field (see also~\cite{AZ-boundary-rigidity-magnetic-and-potential} for a generalization). Magnetic geodesics can be seen as geodesics of a Randers metric under additional assumptions for the vector potential which induces the magnetic field (the magnetic field has to be ``weak")~\cite{HA-spacetimes-randers, TA-magnetic-billiards}. There is also a correspondence between Randers metrics and stationary Lorentzian metrics~\cite{CJM-fermat-metrics, CJS-randers-lorentz, JPRS-randers-lorentz-general} (see~\cite{UYZ-boundary-rigidity-stationary} for a boundary rigidity result on stationary Lorentzian manifolds). 
We note that projectively flat Finsler norms (geodesics of the Finsler norm are segments of straight lines) on compact convex domains in~$\R^2$ are completely determined by their boundary distance data~\cite{ALE-convex-rigidity-plane, AM-pseudo-metrics-on-the-plane, KO-boundary-rigidity-projective-metrics}. In fact, this holds for a more general class of projective metrics in the plane~\cite{KO-boundary-rigidity-projective-metrics}.

Some geometric results similar to the boundary rigidity are known on Finsler manifolds. It was shown in~\cite{deILS-finsler-boundary-distance-map} that the collection of boundary distance maps, which measure distances from the interior to the boundary, determines the topological and differential structures of the Finsler manifold. Further, it was shown in~\cite{deILS-broken-scattering-rigidity} that the broken scattering relation (lengths of all geodesics with endpoints on the boundary and reflecting once in the interior) determines the isometry class of reversible Finsler manifolds admitting a strictly convex foliation.

The boundary rigidity problem is known in seismology as the travel time tomography problem where one tries to recover the speed of sound inside the Earth by measuring travel times of seismic waves on the surface. The ray paths of the seismic waves correspond to geodesics and the travel times correspond to lengths of the geodesics. The travel time tomography problem was solved in the beginning of 20th century for spherically symmetric metrics $g=c^{-2}(r)e$ where~$e$ is the Euclidean metric and $c=c(r)$ is a radial sound speed satisfying the Herglotz condition (see equation \eqref{eq:herglotzcondition})~\cite{HE-inverse-kinematic-problem, WZ-kinematic-problem}. Our results apply to the situation where the seismic wave propagates in a moving medium: one can uniquely recover both the sound speed and the velocity of the medium up to potential fields from travel time measurements (see theorem~\ref{thm:maintheorem} and section~\ref{sec:applications}). The linearization of the boundary rigidity or travel time tomography problem leads to tensor tomography where one wants to characterize the kernel of the geodesic ray transform on symmetric 2-tensor fields~\cite{SHA-ray-transform-riemannian-manifold}. For results in this direction and a general overview of tensor tomography, see the reviews~\cite{IM:integral-geometry-review, PSU-tensor-tomography-progress}. 

\subsection{The main results}\label{sec:mainresults}
Before stating our main results let us briefly introduce some notation; more details can be found in section~\ref{sec:preliminaries}. The proofs of the main theorems can be found in section~\ref{sec:proofsofmaintheorems}.

We denote by~$M$ an $n$-dimensional smooth manifold with boundary~$\partial M$ where $n\geq 2$. We let~$F$ be a Finsler norm and~$F_r$ refers to a reversible Finsler norm, i.e. $F_r(x, -y)=F_r(x, y)$ for all $x\in M$ and $y\in T_xM$. Riemannian metrics are a special case of Finsler norms: if~$g$ is a Riemannian metric, then it induces a reversible Finsler norm~$F_g$ as $F_g(x, y)=\sqrt{g_{ij}(x)y^iy^j}$. We denote by~$\beta$ a smooth closed 1-form ($\der\beta=0$) and $\aabs{\beta}_{F^*}=\sup_{x\in M}F^*(x, \beta(x))$ is the dual norm of~$\beta$ with respect to the co-Finsler norm~$F^*$ in $T^*M$. 

We say that the Finsler norm~$F$ is admissible, if for every two boundary points $x, x'\in\partial M$ there is unique geodesic~$\gamma$ of~$F$ with finite length going from~$x$ to~$x'$. If~$F$ is admissible, then we define the (not necessarily symmetric) map~$d_F(\cdot, \cdot)\colon\partial M\times\partial M\rightarrow [0, \infty)$ by setting $d_F(x, x')=L_F(\gamma)$ where~$L_F(\gamma)$ denotes the length of the curve~$\gamma$ with respect to~$F$. We call the map $d_F(\cdot, \cdot)$ the boundary distance data of~$F$. Finally, we say that the Riemannian manifolds $(M, g_1)$ and $(M, g_2)$ are boundary rigid, if $d_{g_1}(x, x')=d_{g_2}(x, x')$ for all $x, x'\in\partial M$ implies that $g_2=\Psi^*g_1$ where $\Psi\colon M\rightarrow M$ is a diffeomorphism such that $\Psi|_{\partial M}=\id$. In other words, ~$g_1$ and~$g_2$ are isometric as Riemannian metrics.

We recall that a diffeomorphism $\Psi\colon (M, F_2)\rightarrow (M, F_1)$ is an isometry between Finsler manifolds if $\Psi^*F_1=F_2$, or equivalently~$\Psi$ preserves the Finslerian distance~\cite{AK-isometries-finsler}. We make the following observations before giving our first theorem.

\begin{remark}
\label{remark:boundaryrigidityfinsler}
We note that Finsler norms are very flexible with respect to the boundary distance data, i.e. they are not usually boundary rigid in the same sense as Riemannian metrics. Let $\Psi\colon M\rightarrow M$ be a diffeomorphism which is identity on the boundary. If~$F_1$ is an admissible Finsler norm and~$\phi$ is a scalar field which is constant on the boundary and its differential~$\der\phi$ has sufficiently small norm with respect to~$\Psi^*F_1$, then~$F_1$ and $F_2=\Psi^*F_1+\der\phi$ give the same boundary distance data (Finslerian isometries preserve geodesics~\cite{AK-isometries-finsler} and addition of~$\der\phi$ only changes parametrizations of geodesics~\cite{CS-riemann-finsler-geometry}). Especially, if~$F_1$ is reversible, then $\{\Psi^*F_1+\der\phi: \Psi|_{\partial M}=\id \ \text{and} \ \phi|_{\partial M}=\text{constant}\}$ provides a large family of Finsler norms which give the same boundary distances but are not isometric to~$F_1$ (since $\Psi^*F_1+\der\phi$ is non-reversible whenever~$\phi$ is not constant). See also~\cite{BI-boundary-rigidity, CNV-finsler-deformations, CD-length-spectrum-orbifold, IVA-finsler-monotonicity} for results and constructions of non-isometric Finsler norms giving the same boundary distances.
\end{remark}

\begin{remark}
\label{remark:almostisometry}
Finsler norms~$F_1$ and~$F_2$ which satisfy $F_2=\Psi^*F_1+\der\phi$ for some scalar field~$\phi$ and diffeomorphism~$\Psi$ are sometimes called almost isometric Finsler norms and the map $\Psi\colon (M, F_2)\rightarrow (M, F_1)$ is called almost isometry~\cite{CJ-functional-representation-almost-isometries,DJV-lipschitz-almost-isometries, HA-spacetimes-randers, JLP-almost-isometries-spacetimes}. We show in theorem~\ref{thm:maintheorem} that under certain assumptions the boundary distance data determines Randers metrics up to an almost isometry (see also remark~\ref{remark:randersrigidity}). Almost isometries have many good properties: they for example are projective transformations which preserve (minimizing) geodesics up to reparametrization~\cite{JLP-almost-isometries-spacetimes}. Almost isometries can also be defined on general quasi-metric spaces $(X, d)$. It follows that if $\Psi\colon (X_1, d_1)\rightarrow (X_2, d_2)$ is an almost isometry between quasi-metric spaces, then $\Psi$ is an isometry between the metric spaces $(X_1, \widetilde{d}_1)\rightarrow (X_2, \widetilde{d}_2)$ where $\widetilde{d}_i(p, q)=\frac{1}{2}(d_i(p, q)+d_i(q, p))$ is the symmetrized metric~\cite{CJ-functional-representation-almost-isometries, JLP-almost-isometries-spacetimes}. Especially, in the case of metric spaces almost isometries are isometries.
\end{remark}

Our first theorem says that one can uniquely recover (up to potential fields) the perturbation~$\beta$ and the boundary distance data of~$F_r$ from the boundary distance data of~$F=F_r+\beta$. 

\begin{theorem}
\label{thm:sinjectivityfromboundarydistances}
Let~$M$ be a compact and simply connected smooth manifold with boundary. For $i\in\{1, 2\}$ let $F_i=F_{r, i}+\beta_i$ be admissible Finsler norms where~$F_{r, i}$ is an admissible and reversible Finsler norm and~$\beta_i$ is a smooth closed 1-form such that $\aabs{\beta_i}_{F^*_{r, i}}<1$. Then the following are equivalent:
\begin{enumerate}[label=(\roman*)]
    \item\label{item:boundarydatathm1} $d_{F_1}(x, x')=d_{F_2}(x, x')$ for all $x, x'\in \partial M$.
    \medskip
    \item\label{item:gaugethm1} There is unique scalar field~$\phi$ vanishing on the boundary such that $\beta_2=\beta_1+\der\phi$, and $d_{F_{r, 1}}(x, x')=d_{F_{r, 2}}(x, x')$ for all $x, x'\in\partial M$.
\end{enumerate}
\end{theorem}

\begin{remark}
Since~$\beta_i$ is closed and~$M$ is simply connected, it follows that $\beta_i=\der\phi_i$ for some scalar field~$\phi_i$. Thus~$F_{r, i}$ and~$F_i=F_{r, i}+\beta_i=F_{r, i}+\der\phi_i$ are almost isometric (but not isometric) Finsler norms (see remark~\ref{remark:almostisometry}). Trivially one can define $\phi=\phi_2-\phi_1$ so that $\der\phi=\beta_2-\beta_1$. The assumption $d_{F_1}(x, x')=d_{F_2}(x, x')$ for all $x, x'\in \partial M$ is then used to show that~$\phi$ is constant on the boundary (and one can choose this constant to be zero).
\end{remark}

Let us clarify some of our assumptions in theorem~\ref{thm:sinjectivityfromboundarydistances}. We need the assumption $\aabs{\beta_i}_{F^*_{r, i}}<1$ to guarantee that the sum $F_{r, i}+\beta_i$ defines a Finsler norm. Reversibility of~$F_{r, i}$ is needed so that any curve has the same length with respect to~$F_{r, i}$ as any of its reversed reparametrizations. The condition that~$\beta_i$ is closed is used in three places. First, it is equivalent to that~$F_i$ and~$F_{r, i}$ have the same geodesics up to orientation preserving reparametrizations ($F_i$ and~$F_{r, i}$ are projectively equivalent, see lemma~\ref{lemma:projectiveequivalence}). Second, closedness of~$\beta_i$ is also equivalent to that~$F_i$ has reversible geodesics ($F_i$ is projectively reversible, see lemma~\ref{lemma:reversegeodesics}). Third, $\der\beta_i=0$ implies that~$\beta_i$ is exact since~$M$ is assumed to be simply connected. All these properties are in a crucial role in our proofs.

The existence of unique geodesics connecting boundary points is used in the proof as well and for this reason we assume that the Finsler norms are admissible. We note that since~$F_i$ and~$F_{r, i}$ are projectively equivalent, the admissibility of~$F_{r, i}$ implies the admissibility of $F_{i}$, and vice versa. We also note that~$\der\phi$ is closed so the conclusion $\beta_2=\beta_1+\der\phi$ is compatible with the assumptions on~$\beta_i$. The conclusion that~$\beta_i$ differ only by a potential is similar to the solenoidal injectivity result for the geodesic ray transform of 1-forms~\cite{AR-uniqueness-of-one-forms, PSU-tensor-tomography-progress}.

As an application of theorem~\ref{thm:sinjectivityfromboundarydistances} we have the following boundary rigidity result for Randers metrics (see~\cite[Theorem 6.4]{DPSU-boundary-rigidity-magnetic} for a similar result).

\begin{theorem}\label{thm:maintheorem}
Let~$M$ be a compact and simply connected smooth manifold with boundary. For $i\in\{1, 2\}$ let $F_i=F_{g_i}+\beta_i$ be admissible Finsler norms where~$g_i$ is an admissible Riemannian metric and~$\beta_i$ is a smooth closed 1-form such that $\aabs{\beta_i}_{g_i}<1$. Assume that $(M, g_i)$ is boundary rigid. Then the following are equivalent:
\begin{enumerate}[label=(\alph*)]
    \item\label{item:boundarydatathm2} $d_{F_1}(x, x')=d_{F_2}(x, x')$ for all $x, x'\in \partial M$.
    \medskip
    \item\label{item:gaugethm2} There is unique scalar field~$\phi$ vanishing on the boundary and a diffeomorphism~$\Psi$ which is identity on the boundary such that $\beta_2=\beta_1+\der\phi$ and $g_2=\Psi^*g_1$.
    \medskip
    \item\label{item:gauge2thm2} There is unique scalar field~$\phi$ vanishing on the boundary and a diffeomorphism~$\Psi$ which is identity on the boundary such that $\beta_2=\Psi^*\beta_1+\der\phi$ and $g_2=\Psi^*g_1$.
\end{enumerate}
\end{theorem}

\begin{remark}
\label{remark:randersrigidity}
Theorem \ref{thm:maintheorem} part \ref{item:gauge2thm2} implies that $F_2=\Psi^*F_1+\der\phi$, i.e. the Randers metrics~$F_1$ and~$F_2$ are almost isometric (see remark~\ref{remark:almostisometry}). Hence we obtain a boundary rigidity result for Randers metrics in the special case when the 1-form~$\beta$ is closed and the Riemannian metric~$g$ is boundary rigid. This generalizes earlier boundary rigidity results to non-reversible (and hence non-Riemannian) Finsler norms. Note that the diffeomorphism $\Psi\colon (M, F_2)\rightarrow (M, F_1)$ in part~\ref{item:gauge2thm2} is an almost isometry but not an isometry since this would require that $\Psi^*\beta_1=\beta_2$~\cite{BRS-zermelo-navigation}. Also note that if $\beta_1=0$ and $\beta_2\neq 0$, then~$F_1$ and~$F_2$ can not be isometric since~$F_1$ is reversible and~$F_2$ is non-reversible.
\end{remark}

The assumptions of theorem~\ref{thm:maintheorem} are the same as in theorem~\ref{thm:sinjectivityfromboundarydistances} except that we also assume the boundary rigidity of $(M, g)$. We can simultaneously recover the metric~$g$ and the 1-form~$\beta$ from the boundary distance data $d_F(\cdot, \cdot)$ since the reversibility of~$F_g$ implies that the data for~$\beta_i$ and~$g_i$ ``decouple": for any curve~$\gamma$ one can obtain $\int_\gamma\beta$ from the antisymmetric part and~$L_g(\gamma)$ from the symmetric part of the length functional~$L_F(\gamma)$.  We note that in theorems~\ref{thm:sinjectivityfromboundarydistances} and~\ref{thm:maintheorem} we only use the lengths of geodesics connecting boundary points as data.

Admissible Finsler norms as we have defined are closely related to simple Finsler norms and simple Riemannian metrics. A Riemannian metric~$g$ on a smooth manifold~$M$ with boundary is simple if it is non-trapping (geodesics have finite length), geodesics have no conjugate points and the boundary~$\partial M$ is strictly convex with respect to~$g$ (the second fundamental form on~$\partial M$ is positive definite). See~\cite[Section 3.7]{PSU-geometric-inverse-book} for many equivalent definitions of simple Riemannian metrics. The concept of a simple Finsler norm can be defined analogously~\cite{BI-simple-finsler, IVA-finsler-monotonicity}. The simplicity of the Finsler norm or Riemannian metric implies that there exists unique minimizing geodesic between any two points of the manifold~\cite{BI-simple-finsler, IVA-finsler-monotonicity, PSU-geometric-inverse-book}. More generally, if the manifold admits a convex function which has a minimum point, then there is a finite number of geodesics between any two non-conjugate points~\cite{CJM-convex-functions, GMP-convex-functions} (see also~\cite{PSUZ-matrix-weights}). 

We remark that one can take $(M, g_i)$ to be a compact simple surface in theorem~\ref{thm:maintheorem} since simple Riemannian metrics are admissible and in two dimensions they are boundary rigid~\cite{PU-simple-manifolds-boundary-rigidity}. If~$g_1$ and~$g_2$ are simple metrics which are conformal and agree on the boundary, then they are boundary rigid in any dimension $n\geq 2$  \cite{CRO-rigidity-conformal, MU-reconstruction-problem-two-dimensional, SUVZ-travel-time-tomography}.

Theorem~\ref{thm:maintheorem} has an application to Randers metrics arising in seismology (see section~\ref{sec:applications} for more details). Let~$M=\overline{B}(0, R)$ be a closed ball of radius $R>0$ and $g=c^{-2}(r)e$ where~$e$ is the Euclidean metric and~$c=c(r)$ is a radial sound speed satisfying the Herglotz condition
\begin{equation}\label{eq:herglotzcondition}
\frac{\der}{\der r}\bigg(\frac{r}{c(r)}\bigg)>0, \quad r\in [0, R].
\end{equation}
It follows that $(M, g)$ is a non-trapping Riemannian manifold with strictly convex boundary~\cite{MO-herglotz-conjugate-points, SUVZ-travel-time-tomography}. Let us further assume that~$g$ has no conjugate points, i.e.~$g=c^{-2}(r)e$ is a simple Riemannian metric. Then~$g$ is admissible and one can recover~$c$ and hence~$g$ uniquely in theorem~\ref{thm:maintheorem} (see~\cite[Remark 2.10]{PSU-geometric-inverse-book} and~\cite{SHA-ray-transform-riemannian-manifold, SUVZ-travel-time-tomography}). Especially, the diffeomorphism~$\Psi$ becomes identity in this case ($\Psi=\id$ also for general conformal simple metrics which agree on the boundary). However, $\Psi$ can be a nontrivial diffeomorphism for general spherically symmetric Riemannian metrics~$g$ (see~\cite[Appendix C]{deIK-spectral-rigidity}). We also note that there are sound speeds~$c$ satisfying the Herglotz condition~\eqref{eq:herglotzcondition} such that~$g$ has conjugate points (and~$g$ is not admissible anymore, see~\cite[Section 3.3.2 and figure 6]{MO-herglotz-conjugate-points}). In section~\ref{sec:applications} we give a physical interpretation for the 1-form~$\beta$ in theorem~\ref{thm:maintheorem} ($\beta$ corresponds to the flow field of a moving medium).

\subsection{Application in seismology}\label{sec:applications}
Here we give an application of theorems~\ref{thm:sinjectivityfromboundarydistances} and~\ref{thm:maintheorem} to seismology where the seismic wave propagates in a moving medium. Assume that we have an object moving on a Riemannian manifold $(M, g)$ with constant speed $\aabs{U}_g=1$. The speed is fixed, but the object can change the direction of the velocity vector~$U$ arbitrarily. Let~$W$ be a vector field which can be interpreted as the additional velocity resulting from a time-independent external force field acting on the object. The net velocity is $U+W$ and we assume $\aabs{W}_g<1$ so that the object can move freely in any direction.

Given any two points $p, q\in M$ we would like to know which path gives the least time when traveling from~$p$ to~$q$ taking the drift~$W$ into account. This is known as the Zermelo's navigation problem (see \cite{BRS-zermelo-navigation, CS-finsler-geometry-randers, SHE-zermelo-riemannian}). It turns out that the unique solution is given by a geodesic of the Randers metric $F=F_\alpha+\beta$ where (see \cite[Section 2.2]{CS-finsler-geometry-randers})
\begin{align}
\alpha_{ij}&=\frac{g_{ij}}{\lambda}+\frac{W_i}{\lambda}\frac{W_j}{\lambda}, \quad\beta_i=-\frac{W_i}{\lambda} \\
W_i&=g_{ij}W^j, \quad \lambda=1-\aabs{W}_g^2
\end{align}
and we have left the dependence on $x\in M$ implicit. Especially, if the parameter of a piecewise smooth curve $\gamma\colon [0, T]\rightarrow M$ represents time, then (see \cite[Lemma 3.1]{SHE-zermelo-riemannian} and~\cite[Lemma 1.4.1]{CS-riemann-finsler-geometry})
\begin{equation}\label{eq:traveltimefunctional}
T=L_F(\gamma).
\end{equation}

Let us interpret the object as a seismic wave (or ray) propagating in a moving medium. The manifold $M$ corresponds to the Earth which can be modelled as a compact and simply connected smooth manifold with boundary (a ball). By the Fermat's principle the path of the ray is a critical point of the travel time functional~\cite{ABS-seismic-rays-as-finsler-geodesics, BS-fermat-principle, CE-seismic-ray-theory}. But since this functional equals to the length functional~$L_F(\gamma)$ of the Randers metric $F=F_\alpha+\beta$ by  equation~\eqref{eq:traveltimefunctional}, the ray paths of seismic waves correspond to geodesics of~$F$ which is the unique solution to the Zermelo's navigation problem. 

If our Riemannian metric is of the form $g=c^{-2}e$ where~$e$ is the Euclidean metric and $c=c(x)$ is the sound speed, then $\aabs{U}_g=1$ is equivalent to $\aabs{U}_e=c$ where $\aabs{\cdot}_e$ is the Euclidean norm of vectors. Thus~$U$ corresponds to the velocity of the propagating wave and the medium moves with velocity~$W$ for which $\aabs{W}_e<c$. The components of the Randers metric take the form
\begin{align}
\alpha_{ij}&=\frac{c^{-2}\delta_{ij}}{1-c^{-2}\aabs{W}_e^2}+\frac{c^{-4}W^iW^j}{(1-c^{-2}\aabs{W}_e^2)^2} \\
\beta_i&=-\frac{c^{-2}W^i}{1-c^{-2}\aabs{W}_e^2}.
\end{align}
Note that here we have identified $W^i=\delta_{ij}W^j$. Now if the 1-form~$\beta$ is closed, then theorem~\ref{thm:sinjectivityfromboundarydistances} implies that one can uniquely recover~$\beta$ up to potential fields from travel time measurements of seismic waves (assuming admissibility of~$\alpha$). In addition, if the Riemannian manifold $(M, \alpha)$ is boundary rigid, then by theorem~\ref{thm:maintheorem} one can also uniquely recover the Riemannian metric~$\alpha$ up to boundary preserving diffeomorphism from the travel time data.

Let us do the following approximation. If we assume that $\aabs{W}_e/c\ll 1$, then
\begin{align}
\alpha_{ij}&\approx c^{-2}\delta_{ij}+\frac{W^i}{c^2}\frac{W^j}{c^2} \\
\beta_i&\approx -\frac{W^i}{c^2}.
\end{align}
When we only work to first order in $\aabs{W}_e/c$, the Riemannian metric~$\alpha$ reduces to
\begin{align}
\alpha_{ij}\approx c^{-2}\delta_{ij}=g_{ij}
\end{align}
and the ray paths of seismic waves correspond to geodesics of the Randers metric $F=F_g+\beta$. Similar linearization result is obtained in~\cite{GW-geometry-sound-rays-wind} for sound waves propagating in air under the influence of wind. We also note that the same result can be obtained from the linearization of travel time measurements~\cite{NO-tomographic-recostruction-of-vector-fields}.

If the sound speed~$c=c(r)$ is radial, $c$ satisfies the Herglotz condition~\eqref{eq:herglotzcondition} and $g=c^{-2}(r)e$ has no conjugate points, then theorem~\ref{thm:maintheorem} implies that in the first order approximation (with respect to $\aabs{W}_e/c$) one can uniquely recover the sound speed~$c$ and the velocity of the medium~$W$ up to potential fields from travel time measurements. If the speed of sound~$c$ is constant, then the condition $\der(W/c^2)=0$ reduces to $\der W=0$, which in the case of a fluid flow means that~$W$ is irrotational (or curl-free). Note that in the approximation we identify $W^i=\delta_{ij}W^j$. In general, if~$c$ is not constant, then the condition $\der(W/c^2)=0$ only means that the scaled flow field $W/c^2$ is irrotational.

To summarize this section: our results (theorems~\ref{thm:sinjectivityfromboundarydistances} and~\ref{thm:maintheorem}) apply to the propagation of seismic waves in a moving medium. Under certain assumptions one can recover the velocity of the medium from travel time measurements, and at the same time one reduces the travel time tomography problem in moving medium to the case where no flow field is involved. This allows one to recover the speed of sound as well in the first order approximation.
 
\section{Finsler manifolds}\label{sec:preliminaries}

In this section we give a brief introduction to Finsler geometry. We only go through definitions and results which are needed in proving our main theorems. Basic theory of Finsler geometry can be found for example in~\cite{AL-global-aspects-finsler-geometry, BCS-introduction-finsler-geometry,  CS-riemann-finsler-geometry, SHE-lectures-on-finsler-geometry}. We use the Einstein summation convention, i.e. indices which appear both as a subscript and superscript are implicitly summed over.

Let~$M$ be a smooth manifold. We denote by $x\in M$ the base point on the manifold and by $y\in T_x M$ the tangent vectors. A Finsler norm~$F$ on~$M$ is a non-negative function on the tangent bundle $F\colon TM\rightarrow[0, \infty)$ such that
\begin{enumerate}[label=(F\arabic*)]
    \item\label{item:smoothness} $F$ is smooth in $TM\setminus\{0\}$ (smoothness outside zero section)
    \item $F(x, y)=0$ if and only if $y=0$ (positivity)
    \item $F(x, \lambda y)=\lambda F(x, y)$ for all $\lambda\geq 0$ (positive homogeneity of degree 1)
    \item\label{item:convexity} $\frac{1}{2}\frac{\partial^2 F^2(x, y)}{\partial y^i\partial y^j}$ is positive definite whenever $y\neq 0$ (convexity). 
\end{enumerate}
The pair $(M, F)$ is called a Finsler manifold. If~$F$ is a Finsler norm, then one can define the reversed Finsler norm~$\overleftarrow{F}$ by setting~$\overleftarrow{F}(x, y)=F(x, -y)$. It follows that~$\overleftarrow{F}$ also satisfies the properties \ref{item:smoothness}--\ref{item:convexity}.

The conditions \ref{item:smoothness}--\ref{item:convexity} imply that for every $x\in M$ the map $y\mapsto F(x, y)$ defines a positively homogeneous norm in $T_xM$. If $F(x, -y)=F(x, y)$ for all $x\in M$ and $y\in T_xM$, we say that the Finsler norm~$F$ is reversible (or absolutely homogeneous). If~$F$ is reversible, then the map $y\mapsto F(x, y)$ defines a norm in $T_xM$. Every Riemannian metric $g=g(x)$ on~$M$ induces a reversible Finsler norm~$F_g$ on~$M$ by setting
\begin{equation}\label{eq:riemannianfinslernorm}
F_g(x, y)=\sqrt{g_{ij}(x)y^iy^j}.
\end{equation}

The condition \ref{item:convexity} allows us to define the local metric $g_{ij}=g_{ij}(x, y)$ as
\begin{equation}
g_{ij}(x, y)=\frac{1}{2}\frac{\partial^2 F^2(x, y)}{\partial y^i\partial y^j}.
\end{equation}
One can then define the Legendre transformation $L\colon TM\rightarrow T^*M$ using the local metric~$g_{ij}$ (see for example~\cite[Chapter 3.1]{SHE-lectures-on-finsler-geometry}). If~$F\colon TM\rightarrow[0, \infty)$ is a Finsler norm, then by using the Legendre transformation one obtains the dual norm (or co-Finsler norm) $F^*\colon T^*M\rightarrow [0, \infty)$ satisfying the properties \ref{item:smoothness}--\ref{item:convexity} in~$T^*M$. The dual norm of a covector $\omega\in T^*_xM$ becomes
\begin{equation}
F^*(x, \omega)=\sup_{\substack{y\in T_xM \\F(x, y)=1}}\omega(y).
\end{equation}
If $F=F_g$ where $g=g(x)$ is a Riemannian metric, then $g_{ij}(x, y)=g_{ij}(x)y^iy^j$ and the Legendre transformation~$L$ and its inverse correspond to the musical isomorphisms.

In this article we study a class of non-reversible Finsler norms. Let~$F_1$ be a Finsler norm on~$M$ and~$\beta$ a smooth nonzero 1-form on~$M$. Assume that the dual norm of~$\beta$ satisfies
\begin{equation}
\aabs{\beta}_{F^*_1}:=\sup_{x\in M}F_1^*(x, \beta(x))<1.
\end{equation}
Then $F=F_1+\beta$ defines also a Finsler norm on~$M$ (see~\cite[Example 6.3.1]{SHE-lectures-on-finsler-geometry} and~\cite[Chapter 11.1]{BCS-introduction-finsler-geometry}). We study the special case $F=F_r+\beta$ where~$F_r$ is a reversible Finsler norm. It follows that Finsler norms of this kind are non-reversible since $F(x, -y)=F(x, y)$ for all $x\in M$ and $y\in T_xM$ if and only if $\beta\equiv 0$. If $F_r=F_g$ where~$g$ is a Riemannian metric, then $F=F_g+\beta$ is called a Randers metric (see~\cite{RAN-randers-metric} for the original definition of a Randers metric). Randers metrics are examples of Finsler norms which are not induced by any Riemannian metric (since Riemannian metrics are always reversible).

The length of a piecewise smooth curve $\gamma\colon [a, b]\rightarrow M$ is defined to be 
\begin{equation}
L_F(\gamma)=\int_a^b F(\gamma(t), \dot{\gamma}(t))\der t.
\end{equation}
In general, $L_F(\gamma)$ is invariant only in orientation preserving reparametrizations. If in addition~$F$ is reversible, then~$L_F(\gamma)$ is also invariant in orientation reversing reparametrizations.
When~$F$ is induced by a Riemannian metric~$g$, then we simply write $L_g:=L_{F_g}$. If~$F$ is a Finsler norm such that $F=F_1+\beta$ where $F_1$ is a Finsler norm and~$\beta$ is a 1-form, then for any piecewise smooth curve~$\gamma$ we have
\begin{equation}
L_F(\gamma)=L_{F_1}(\gamma)+\int_\gamma\beta.
\end{equation}
Note that for the term coming from the 1-form~$\beta$ we have 
\begin{equation}
\int_{\widetilde{\gamma}}\beta=\pm\int_\gamma\beta
\end{equation}
where the plus sign corresponds to reparametrizations~$\widetilde{\gamma}$ of~$\gamma$ preserving the orientation and the minus sign corresponds to reparametrizations reversing the orientation. 

A smooth curve~$\gamma$ on $(M, F)$ is a geodesic, if it satisfies the geodesic equation
\begin{equation}
\ddot{\gamma}^i(t)+2G^i(\gamma(t), \dot{\gamma}(t))=0
\end{equation}
where the spray coefficients $G^i=G^i(x, y)$ are given by
\begin{equation}
G^i(x, y)=\frac{1}{4}g^{il}(x, y)\bigg(y^k\frac{\partial^2 F^2(x, y)}{\partial x^k\partial y^l}-\frac{\partial F^2(x, y)}{\partial x^l}\bigg).
\end{equation}
Here $g^{ij}(x, y)$ are the components of the inverse matrix of $g_{ij}(x, y)$. Geodesics correspond to straightest possible paths in $(M, F)$ and they are locally minimizing. Geodesics are also critical points of the length functional~$L_F(\gamma)$. 

We say that two Finsler norms~$F_1$ and~$F_2$ on a smooth manifold~$M$ are projectively equivalent, if~$F_1$ and~$F_2$ have the same geodesics as point sets. More precisely,~$F_1$ and~$F_2$ are projectively equivalent, if for any geodesic~$\gamma$ of~$F_1$ there is an orientation preserving reparametrization~$\eta$ of~$\gamma$ such that~$\eta$ is a geodesic of~$F_2$, and vice versa.
We also say that a Finsler norm~$F$ has reversible geodesics (or is projectively reversible), if for any geodesic~$\gamma$ of~$F$ the reversed curve~$\overleftarrow{\gamma}$ is also a geodesic of~$F$ up to orientation preserving reparametrization. In other words, $F$ has reversible geodesics if~$F$ and~$\overleftarrow{F}$ are projectively equivalent.

In general, if~$\gamma$ is a geodesic of~$F$, then the reversed curve~$\overleftarrow{\gamma}$ is not necessarily a geodesic of~$F$. If~$F$ is reversible, then~$\overleftarrow{\gamma}$ is also a geodesic. The following lemma says that the same holds (modulo orientation preserving reparametrization) if we perturb a reversible Finsler norm with a closed 1-form (see also  \cite[Theorem 7.1]{MSS-reversible-geodesics} for a more general version of the lemma).

\begin{lemma}[{\cite[p. 406]{CRA-randers-reversible-geodesics}}]\label{lemma:reversegeodesics}
Let~$F=F_r+\beta$ be a Finsler norm where~$F_r$ is a reversible Finsler norm and~$\beta$ is a 1-form such that $\aabs{\beta}_{F^*_r}<1$. Then~$F$ has reversible geodesics (is projectively equivalent to~$\overleftarrow{F}$) if and only if~$\beta$ is closed ($\der\beta=0$).
\end{lemma}

The next lemma has a central role in the proofs of our main theorems. It says that if we perturb a Finsler norm with a closed 1-form, then the geodesics change only by an orientation preserving reparametrization (see also~\cite[Theorem 3.3]{CA-projective-equivalence-finsler} and~\cite[Example 2.11]{AL-global-aspects-finsler-geometry}).

\begin{lemma}[{\cite[Theorem 3.3.1 and Example 3.3.2]{CS-riemann-finsler-geometry}}]\label{lemma:projectiveequivalence}
Let~$F_1$ be a Finsler norm on a smooth manifold~$M$. Let $F_2=F_1+\beta$ be another Finsler norm where~$\beta$ is a 1-form such that $\aabs{\beta}_{F^*_1}<1$. Then $(M, F_1)$ and $(M, F_2)$ are projectively equivalent if and only if~$\beta$ is closed ($\der\beta$=0).
\end{lemma}

\section{Proofs of the main theorems}
\label{sec:proofsofmaintheorems}
In this section we prove our main results. The proofs are based on lemmas~\ref{lemma:reversegeodesics} and~\ref{lemma:projectiveequivalence} which imply that~$F_i$ and~$F_{r, i}$ have the same geodesics up to orientation preserving reparametrizations and that~$F_i$ has reversible geodesics. This allows us to express the integrals of the 1-forms~$\beta_i$ in terms of the boundary distance data of~$F_i$. Similarly we can express the boundary distance data of~$g_i$ in terms of the boundary distance data of~$F_i$, which implies that~$g_1$ and~$g_2$ differ only by a boundary preserving diffeomorphism since the underlying manifolds $(M, g_i)$ are assumed to be boundary rigid. 

We are now ready to prove our main theorems. Recall that a Finsler norm~$F$ is admissible if every two boundary points can be joined by unique geodesic of~$F$ with finite length.

\begin{proof}[Proof of theorem \ref{thm:sinjectivityfromboundarydistances}.]
Let us first prove the direction \ref{item:boundarydatathm1}$\Rightarrow$\ref{item:gaugethm1}. We note that if~$\gamma$ is any curve on~$M$ and~$\overleftarrow{\gamma}$ any of its reversed reparametrizations, then reversibility of~$F_{r, i}$ implies that $L_{F_{r, i}}(\gamma)=L_{F_{r, i}}(\overleftarrow{\gamma})$ and
\begin{equation}
\int_\gamma\beta_i=\frac{L_{F_i}(\gamma)-L_{F_i}(\overleftarrow{\gamma})}{2}.
\end{equation}
Let $x, x'\in\partial M$ and~$\gamma_i$ be the unique geodesic of~$F_{r, i}$ connecting~$x$ to~$x'$ (see figure~\ref{fig:proof}). By lemma~\ref{lemma:projectiveequivalence} there is an orientation preserving reparametrization~$\eta_i$ of~$\gamma_i$ such that~$\eta_i$ is a geodesic of~$F_i$. Note that since~$\eta_i$ connects~$x$ to~$x'$ we have $L_{F_i}(\eta_i)=d_{F_i}(x, x')$ by  admissibility of~$F_i$. Using lemma~\ref{lemma:reversegeodesics} let~$\overleftarrow{\eta_i}$ be the reversed curve which is a geodesic of~$F_i$. Now again by admissibility of~$F_i$ we have $L_{F_i}(\overleftarrow{\eta_i})=d_{F_i}(x', x)$ since~$\overleftarrow{\eta_i}$ connects~$x'$ to~$x$. We obtain
\begin{align}
\int_{\gamma_1}\beta_1&=\int_{\eta_1}\beta_1=\frac{L_{F_1}(\eta_1)-L_{F_1}(\overleftarrow{\eta_1})}{2}=\frac{d_{F_1}(x, x')-d_{F_1}(x', x)}{2} \\ &=\frac{d_{F_2}(x, x')-d_{F_2}(x', x)}{2}=\frac{L_{F_2}(\eta_2)-L_{F_2}(\overleftarrow{\eta_2})}{2}=\int_{\eta_2}\beta_2=\int_{\gamma_2}\beta_2.
\end{align}
The closed 1-form~$\beta_i$ is exact because~$M$ is simply connected, i.e. $\beta_i=\der\phi_i$ for some scalar field~$\phi_i$. Since~$\gamma_1$ and~$\gamma_2$ both connect~$x$ to~$x'$, we obtain that
\begin{equation}
\phi_1(x')-\phi_1(x)=\int_{\gamma_1}\beta_1=\int_{\gamma_2}\beta_2=\phi_2(x')-\phi_2(x).
\end{equation}
It follows that $\phi_2-\phi_1$ is constant on the boundary. Let this constant be~$c\in\R$ and define the scalar field $\phi=\phi_2-\phi_1-c$. Then~$\phi$ satisfies $\der\phi=\beta_2-\beta_1$ and $\phi|_{\partial M}=0$. If there is another scalar field~$\phi'$ such that $\der\phi'=\beta_2-\beta_1$ and $\phi'|_{\partial M}=0$, then $\der(\phi-\phi')=0$ and $\phi-\phi'=\text{constant}=0$ since~$M$ is connected and both scalar fields vanish on the boundary. This proves the first claim of the first implication.

For the second claim we note that for any curve~$\gamma$ and any of its reversed reparametrization~$\overleftarrow{\gamma}$ it holds that 
\begin{equation}
L_{F_{r, i}}(\gamma)=\frac{L_{F_i}(\gamma)+L_{F_i}(\overleftarrow{\gamma})}{2}.
\end{equation}
Now let~$\gamma_i$ and~$\eta_i$ be as in the beginning of the proof. It follows that
\begin{align}
d_{F_{r, 1}}(x, x')&=L_{F_{r, 1}}(\gamma_1)=L_{F_{r, 1}}(\eta_1)=\frac{L_{F_1}(\eta_1)+L_{F_1}(\overleftarrow{\eta_1})}{2} \\&=\frac{d_{F_1}(x, x')+d_{F_1}(x', x)}{2} =\frac{d_{F_2}(x, x')+d_{F_2}(x', x)}{2} \\ &= \frac{L_{F_2}(\eta_2)+L_{F_2}(\overleftarrow{\eta_2})}{2}=L_{F_{r, 2}}(\eta_2)=L_{F_{r, 2}}(\gamma_2)=d_{F_{r, 2}}(x, x')
\end{align}
proving the second claim of the first implication.

Let us then prove the implication \ref{item:gaugethm1}$\Rightarrow$\ref{item:boundarydatathm1}. First we note that for any curve~$\gamma$ it holds that
\begin{equation}
L_{F_i}(\gamma)=L_{F_{r, i}}(\gamma)+\int_\gamma\beta_i.
\end{equation}
Let $x, x'\in\partial M$ and~$\eta_i$ be the unique geodesic of~$F_i$ connecting~$x$ to~$x'$. By lemma~\ref{lemma:projectiveequivalence} there is an orientation preserving reparametrization~$\sigma_i$ of~$\eta_i$ such that~$\sigma_i$ is a geodesic of~$F_{r, i}$. Simply connectedness of~$M$ and the assumptions on~$\beta_i$ imply that $\int_{\eta_2}\beta_2=\int_{\eta_1}\beta_1$. Using the assumption $d_{F_{r, 1}}(x, x')=d_{F_{r, 2}}(x, x')$ and the admissibility of~$F_{r, i}$ it follows that
\begin{align}
d_{F_2}(x, x')&=L_{F_2}(\eta_2)=L_{F_{r, 2}}(\eta_2)+\int_{\eta_2}\beta_2=L_{F_{r, 2}}(\sigma_2)+\int_{\eta_1}\beta_1 \\ 
&=d_{F_{r, 2}}(x, x')+\int_{\eta_1}\beta_1=d_{F_{r, 1}}(x, x')+\int_{\eta_1}\beta_1 \\ 
&=L_{F_{r, 1}}(\sigma_1)+\int_{\eta_1}\beta_1=L_{F_{r, 1}}(\eta_1)+\int_{\eta_1}\beta_1 \\
&= L_{F_1}(\eta_1)=d_{F_1}(x, x').
\end{align}
This concludes the proof.
\end{proof}

\begin{figure}[htp]
\centering
\includegraphics[height=7.8cm]{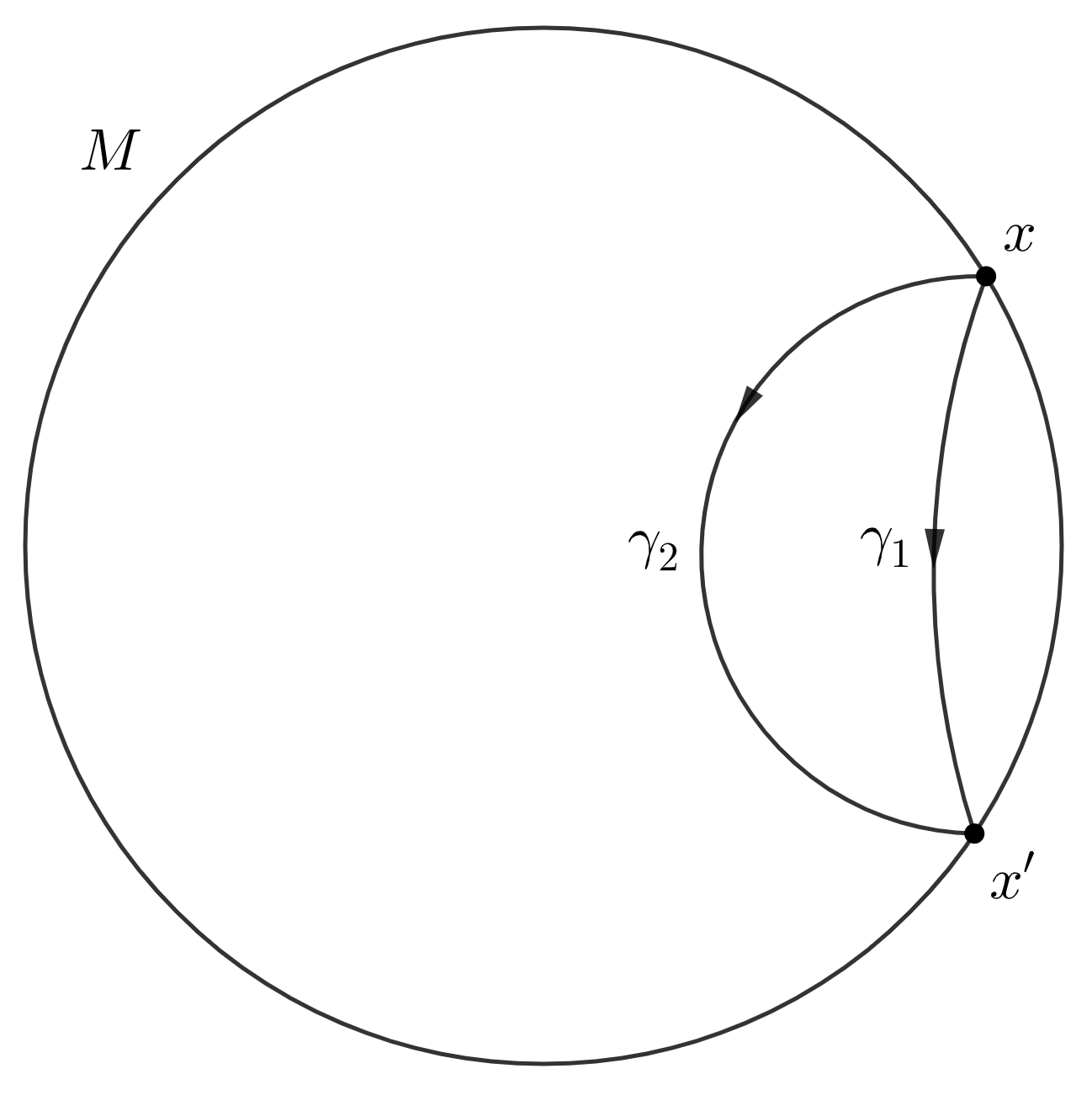}
\caption{A picture illustrating the proof of theorem~\ref{thm:sinjectivityfromboundarydistances}. Here~$\gamma_i$ is a geodesic of~$F_{r, i}$ connecting the boundary point $x\in \partial M$ to the boundary point $x'\in\partial M$. The curves~$\gamma_i$ are also geodesics of~$F_i$ up to orientation preserving reparametrization by lemma~\ref{lemma:projectiveequivalence}. The picture is highly simplified; in reality the curves~$\gamma_1$ and~$\gamma_2$ can for example cross each other.}
\label{fig:proof}
\end{figure}

\begin{proof}[Proof of theorem \ref{thm:maintheorem}.]
If $d_{F_1}(x, x')=d_{F_2}(x, x')$ for all $x, x'\in\partial M$, then by theorem~\ref{thm:sinjectivityfromboundarydistances} there is unique scalar field~$\phi$ vanishing on the boundary such that $\beta_2=\beta_1+\der\phi$, and $d_{g_1}(x, x')=d_{g_2}(x, x')$ for all $x, x'\in\partial M$. Since we assume that $(M, g_i)$ are boundary rigid, there is a diffeomorphism $\Psi\colon M\rightarrow M$ which is identity on the boundary such that $g_2=\Psi^*g_1$. This proves the implication \ref{item:boundarydatathm2}$\Rightarrow$\ref{item:gaugethm2}. The implication \ref{item:gaugethm2}$\Rightarrow$\ref{item:boundarydatathm2} is proved in the same way as the implication \ref{item:gaugethm1}$\Rightarrow$\ref{item:boundarydatathm1} in theorem~\ref{thm:sinjectivityfromboundarydistances} using the fact that~$\Psi$ is a Riemannian isometry fixing boundary points.

Let us then prove the equivalence \ref{item:gaugethm2}$\Leftrightarrow$\ref{item:gauge2thm2}. Let $\Psi\colon M\rightarrow M$ be a diffeomorphism which is identity on the boundary. Since~$\beta_1$ is closed and the pullback commutes with the differential, we have that~$\Psi^*\beta_1$ is also closed. This implies that $\beta_1-\Psi^*\beta_1$ is closed and hence exact because~$M$ is simply connected, i.e. there is a scalar field~$\widetilde{\phi}$ such that $\beta_1-\Psi^*\beta_1=\der\widetilde{\phi}$. Let $x, x'\in\partial M$ be any two boundary points and~$\gamma$ any curve connecting~$x$ to~$x'$. Since~$\Psi$ is identity on the boundary and~$\beta_1$ is exact we have
\begin{equation}
0=\int_\gamma(\beta_1-\Psi^*\beta_1)=\int_\gamma\der\widetilde{\phi}=\widetilde{\phi}(x')-\widetilde{\phi}(x).
\end{equation}
Therefore~$\widetilde{\phi}$ is constant on the boundary and we can subtract this constant to obtain a scalar field~$\phi'$ such  that $\beta_1-\Psi^*\beta_1=\der\phi'$ and $\phi'|_{\partial M}=0$. Thus~$\beta_1$ and~$\Psi^*\beta_1$ differ only by a potential which vanishes on the boundary, concluding the proof.
\end{proof}

\subsection*{Acknowledgements}
The author was supported by Academy of Finland (Centre of Excellence in Inverse Modelling and Imaging, grant numbers 284715 and 309963).
The author is grateful to Joonas Ilmavirta for helpful discussions and suggestions to improve the article. The author wants to thank Jesse Railo and Teemu Saksala for discussions, and \'Arp\'ad Kurusa for providing access to the article~\cite{KO-boundary-rigidity-projective-metrics}. The author also wishes to thank the anonymous referee for helpful comments.
\bibliography{refs} 
\bibliographystyle{abbrv}
\end{document}